\documentclass[10pt]{article}
\usepackage[a4paper]{anysize}\marginsize{1.5cm}{1.5cm}{1.5cm}{1cm}
\pdfpagewidth=\paperwidth \pdfpageheight=\paperheight
\usepackage{amsfonts,amssymb,amsthm,amsmath,eucal,tabu}
\usepackage{pgf}
 \usepackage{array}
 \usepackage{pstricks}
 \usepackage{pstricks-add}
 \usepackage{pgf,tikz}
 \usetikzlibrary{automata}
 \usetikzlibrary{arrows}
 \usepackage{indentfirst}
\theoremstyle{plain}
\newtheorem{thm}{Theorem}[section]
\newtheorem{theorem}[thm]{Theorem}
\newtheorem*{theoremA}{Theorem A}

\newtheorem{conjecture}[thm]{Conjecture}

\theoremstyle{definition}
\newtheorem{definition}[thm]{Definition}
\newtheorem{remark}[thm]{Remark}

\newtheorem{thevarthm}[thm]{\varthmname}

\newenvironment{varthm*}[1]{\trivlist\item[]{\bf #1.}\it}{\endtrivlist}

\renewcommand\geq{\geqslant}

\renewcommand\leq{\leqslant}

\newcommand\be{\begin{eqnarray*}}
\newcommand\ee{\end{eqnarray*}}

\newcommand\newop[2]{\def#1{\mathop{\rm #2}\nolimits}}
\newop\log{log}
\newop\ord{ord}
\newop\Gal{Gal}
\newop\SL{SL}
\newop\Bl{Bl}
\newop\mult{mult}
\newop\mass{mass}
\newop\div{div}
\newop\codim{codim}
\newop\sing{sing}
\newop\vdim{vdim}
\newop\edim{edim}
\newop\Ass{Ass}
\newop\size{size}
\newop\reg{reg}
\newop\satdeg{satdeg}
\newop\supp{supp}
\newop\Neg{Neg}
\newop\Nef{Nef}
\newop\Nefh{Nef_H}
\newop\Eff{Eff}
\newop\Zar{Zar}
\newop\MB{MB}
\newop\MBxC{MB\mathit{(x,C)}}
\newop\NnB{NnB}
\newop\Bigg{Big}
\newop\Effbar{\overline{\Eff}}

\def\keywordname{{\bfseries Keywords}}%
\def\keywords#1{\par\addvspace\medskipamount{\rightskip=0pt plus1cm
\def\and{\ifhmode\unskip\nobreak\fi\ $\cdot$
}\noindent\keywordname\enspace\ignorespaces#1\par}}
\def\subclassname{{\bfseries Mathematics Subject Classification
(2000)}\enspace}
\def\subclass#1{\par\addvspace\medskipamount{\rightskip=0pt plus1cm
\def\and{\ifhmode\unskip\nobreak\fi\ $\cdot$
}\noindent\subclassname\ignorespaces#1\par}}

\begin{document}
\title{Harbourne constants and conic configurations on the projective plane}
\author{Piotr Pokora, Halszka Tutaj-Gasi\'nska}

\date{\today}
\maketitle
\thispagestyle{empty}
\begin{abstract}
In this note we exhibit the so-called Harbourne constants which capture and measure the bounded negativity on various birational models of an algebraic
surface. We show an estimation for Harbourne constants for conic configurations on the complex projective plane with transversal intersection points.
\subclass{14C20}
\end{abstract}

\section{Introduction}

In this note we would like to focus on Harbourne constants in the context of the bounded negativity conjecture for smooth projective surfaces.
\begin{definition}
Let $X$ be a smooth projective surface. We say that $X$ has the \emph{bounded negativity property} if there exists an integer $b(X)$ such that for \emph{every reduced curve} $C \subset X$ one has the bound
$$C^{2} \geq -b(X).$$
\end{definition}
The bounded negativity conjecture (BNC for short) is one of the
most intriguing problem in the theory of projective surfaces and
can be formulated as follows.
\begin{conjecture}[BNC]
An arbitrary smooth complex projective surface has the bounded negativity property.
\end{conjecture}
It is not difficult to see that the BNC is not true in the case of positive characteristic (see for instance \cite[Remark I.2.2.]{Harbourne1}) thus through this paper \emph{we will work over the complex numbers}.

Some surfaces are known to have the bounded negativity property,
for instance those with the $\mathbb{Q}$-effective anticanonical divisor,
$K3$ surfaces, Enriques surfaces and abelian surfaces -- see
\cite{Duke, Harbourne1} for detailed explanations. However, it is
not clear when the BNC holds if one blows up points of those
surfaces. In particular, it is not know whether the blow-up of
$\mathbb{P}^2$ in ten general points has bounded negativity
property.

Let us now recall some recent results which outline a significant progress around the BNC. In \cite{BdRHHLPSz} the authors show the following theorem which is also the very first explicit estimation of self-intersection numbers for certain classes of curves.
\begin{theorem}
Let $f: X_{s} \rightarrow \mathbb{P}^2$ be the blow-up of
$\mathbb{P}^{2}$ in $s$ mutually distinct points. Let $\mathcal{L}
\subset \mathbb{P}^{2}$ be a line configuration on
$\mathbb{P}^{2}$. Denote by $L$ the strict transform of
$\mathcal{L}$. Then $L^{2} \geq -4\cdot s$.
\end{theorem}
In the above theorem the constant $4s$ gives a quite strict
estimation -- the most negative value of the self-intersection we
know is delivered by Wiman configuration (see the last section of the note). Moreover, there is no assumption that we blow up
the projective plane along singular loci of line configurations.
The proof of the above theorem is based on the Hirzebruch
inequality \cite{Hirzebruch} for line configurations and the
so-called Harbourne constants, which will be introduced in the
next section. In \cite{Xavier} the author studies the BNC in the
case of blow-ups of abelian surfaces and blow-ups of the
projective plane and using a certain family of elliptic curves
from \cite{Xavier1} he shows the following interesting result.
\begin{theorem}
Let $\mathcal{C}_{n}$ be a family of elliptic curves on the
complex projective plane which comes from \cite[Section
3]{Xavier1} and let $f: X_{s} \rightarrow
\mathbb{P}^{2}_{\mathbb{C}}$ be the blow-up along the set of $s$
singular points ${\rm Sing}(\mathcal{C}_{n})$. Denote by $C_{n}$
the strict transform of $\mathcal{C}_{n}$. Then one has
$${\rm lim}_{n \rightarrow \infty} \frac{C_{n}^{2}}{s} = -4.$$
\end{theorem}
At the end of this part, let us also recall that in \cite{Duke,
MT} the authors study the BNC for Shimura curves and they obtain
some surprising finiteness results for such curves.

In this paper we will consider the BNC for conic configurations (we assume that all conics are irreducible) \emph{having only transversal intersection points} on the complex projective plane and we show
for them the bounded negativity property. Namely, we prove the following:
\begin{theoremA}
Let $\mathcal{C}$ be a conic configuration of degree $2k$ having only transversal intersection points
 and such that there is no point through which all conics pass. Let $f: X_{s} \rightarrow \mathbb{P}^2$ be the blow-up along ${\rm Sing}(\mathcal{C}) = \{P_{1}, ..., P_{s}\}$.
  Denote by $C$ the strict transform of $\mathcal{C}$, then one has $C^{2} \geq -4.5s$.
\end{theoremA}
Moreover, we consider also the situation when all conics have
$2,3$ or $4$ points in common. Using then the standard Cremona
transformation, and passing to $\mathbb{P}^{1} \times
\mathbb{P}^{1}$ if necessary, we are able to give a bound on $C^2$
also in these cases. As an auxiliary result we prove a
Hirzebruch-type inequality for irreducible $(1,1)$-curve
configurations on $\mathbb{P}^{1} \times \mathbb{P}^{1}$. In
addition, we show that using the standard Cremona transform it is
possible to construct interesting conic configurations from the
well-known line configurations due to Klein and Wiman.

\section{Bounded Negativity viewed by Harbourne Constants}
We start with introducing the Harbourne constants \cite{BdRHHLPSz}.

\begin{definition}\label{def:H-constants}
   Let $X$ be a smooth projective surface and let
   $\mathcal{P}=\{ P_{1},\ldots,P_{s} \}$ be a set of mutually
   distinct $s \geq 1$ points in $X$. Then the \emph{local $H$--constant of $X$ at $\mathcal{P}$}
   is defined as
   \begin{equation}\label{eq:H-const for calp}
      H(X;\mathcal{P}):= \inf \frac{\left(f^*C-\sum_{i=1}^s \mult_{P_i}C\cdot E_i\right)^2}{s},
   \end{equation}
   where $f: Y \to X$ is the blow-up of $X$ at the set $\mathcal{P}$
   with exceptional divisors $E_{1},\ldots,E_s$ and the infimum is taken
   over all \emph{reduced} curves $C\subset X$.\\
   Similarly, we define the \emph{$s$--tuple $H$--constant of $X$}
   as
   $$H(X;s):=\inf_{\mathcal{P}}H(X;\mathcal{P}),$$
   where the infimum now is taken over all $s$--tuples of mutually
   distinct points in $X$. \\
   Finally, we define the \emph{global $H$--constant of $X$} as
   $$H(X):=\inf_{s \geq 1}H(X;s).$$
\end{definition}
   The relation between $H$--constants and the BNC can be expressed as follows.
   Suppose that $H(X)$ is a real number.
   Then for any $s \geq 1$ and any reduced curve $D$ on the blow-up of $X$
   in $s$ points, we have
   $$D^2 \geq sH(X).$$
   Hence BNC holds on all blow-ups of $X$ at $s$ mutually distinct points with the constant $b(X) = sH(X)$.
   On the other hand, even if $H(X)=-\infty$, the BNC might still be true.

   \section{The bounded negativity property for conic configurations}

  Given a configuration of conics on $\mathbb{P}^{2}_{\mathbb{C}}$ we denote by $t_{r}$
  the number of its $r$-fold  points -- i.e. points at which exactly $r$ conics of the configuration meet transversally. In the sequel we will use frequently
  that $\sum_{i} {\rm mult}_{P_{i}}(C) = \sum_{r \geq 2}rt_{r}$ and the notation $H_{\mathcal{C}}$ for conical
  H-constants, for instance $H_{C}(\mathbb{P}^{2}_{\mathbb{C}};{\rm Sing}(\mathcal{C}))$ denotes the conical
  local H-constant at the singular locus ${\rm Sing}(\mathcal{C})$, where the infimum is taken over reduced configurations of conics with transversal intersection points.
  Recall that we have the following combinatorial equality for configurations of conics
   \begin{equation}
   \label{combinatorial}
   4 { k \choose 2} = \sum_{r \geq 2} { r \choose 2} t_{r}.
   \end{equation}
   Moreover, for $i \in \{0,1,2\}$ we define $f_{i} = \sum_{r \geq 2} r^{i} t_{r}$.
   In order to get a lower bound for conical $H$-constant we will use the following result, which comes from \cite{LT}.

   \begin{theorem}
   \label{wielomian}
   Let $\mathcal{C}$ be a configuration of $k \geq 3$ conics with transversal intersections and $t_{k} = 0$.
   Let us define the polynomial
   $$F(x) = (2k+f_{0})x^2 + 2(3k-f_{1}+2f_{0})x + 4(f_{0}-t_{2}).$$
   Then $F(x) \geq 0$ for all $x \in \mathbb{Z}$.
   \end{theorem}
   Now we are ready to proof the first main result.
\begin{theorem} Let $\mathcal{C}$ be a conic configuration of degree $2k$ having transversal intersection points
 and $t_{k} = 0$. Let $f: X_{s} \rightarrow \mathbb{P}^2$ be the blow-up along ${\rm Sing}(\mathcal{C}) = \{P_{1}, ..., P_{s}\}$. Denote by $C$ the strict transform of $\mathcal{C}$. Then one has $C^{2} \geq -4.5s$.
\end{theorem}

\begin{proof} To avoid trivialities we assume that $k \geq 3$. Our aim is to show that for a conic configuration with $t_{k} = 0$ one has
$$H_{C}(\mathbb{P}^2; {\rm Sing}(\mathcal{C})) = \frac{(2k)^2 - \sum_{r \geq 2} r^{2}t_{r}}{\sum_{r \geq 2} t_{r}} \geq -\frac{9}{2}.$$
Using combinatorial equality (\ref{combinatorial}) we obtain
$$ 4k^{2} - \sum_{r \geq 2} r^{2} t_{r} = 4k - \sum_{r \geq 2} r t_{r}.$$
This leads to
$$H_{C}(\mathbb{P}^2; {\rm Sing}(\mathcal{C})) = \frac{4k - \sum_{r\geq 2}rt_{r}}{ \sum_{r \geq 2}t_{r}}.$$
Now we are going to use Theorem \ref{wielomian}. For $x=1$ we obtain
$$8k - 2f_{1} - 4t_{2} + 9f_{0} \geq 0.$$
After simple manipulations we get the following sequence of inequalities
$$8k - 2f_{1} \geq -9f_{0} + 4t_{2} \geq -9f_{0},$$
and finally
$$H_{C}(\mathbb{P}^2; {\rm Sing}(\mathcal{C})) = \frac{4k - \sum_{r \geq 2}rt_{r} }{\sum_{r \geq 2}t_{r}} = \frac{4k - f_{1}}{f_{0}} \geq - \frac{9}{2},$$
which ends the proof.
\end{proof}
Now we want to show that the bounded negativity property holds for configurations of conics with $t_{k} \neq 0$. We exhibit three cases, namely:
\begin{itemize}
\item $t_{k} = 4$. This case corresponds to the one parameter family of conics passing through $4$ points and we have $H_{C}(\mathbb{P}^{2}; {\rm Sing}(\mathcal{C}))= 0$.
\item $t_{k} = 3$. We will use the standard Cremona transform based on these three $t_{k}$ points, let us denote them by $P_1,P_2,P_3$. Note that these
 points are not collinear since our conics are irreducible. In what follows we will denote points before and after the standard Cremona transform by the same letters,
 hoping that it does not lead to misunderstanding. After performing Cremona
 transform each conic passing through points $P_1,P_2,P_3$ and through other points $P_4,...,P_s$ goes to a line, which omits (the
 new) $P_1,P_2,P_3$,
 and passes through (the images of) $P_4,...,P_s$. Denote by small letters $f_0, f_1$ the constants for the considered configurations of conics and by capital
letters $F_0,F_1$ these constants for the resulting configurations of lines. Observe that  $F_0=f_0-3$ and $F_1=f_1-3k$. Thus
$$H_{C}(\mathbb{P}^2;  {\rm Sing}(\mathcal{C})) = \frac{4k-f_1}{f_0}=\frac{4k-3k-F_1}{F_0+3}=\frac{k-F_1}{F_0+3}.$$
By \cite[Theorem 3.3]{BdRHHLPSz} we know that $\frac{k-F_1}{F_0}>-4$. In the case $k-F_1\geq 0$ we have of course
$$\frac{k-F_1}{F_0+3}>-4,$$
and for  $k-F_1<0$ we have
$$\frac{k-F_1}{F_0+3}> \frac{k-F_1}{F_0}>-4,$$
so surely $$H_{C}(\mathbb{P}^2; {\rm Sing}(\mathcal{C}))>-\frac92.$$
\item $t_{k}=2$.
 Now we need to find an estimation of the following $H$-constant
 $$H_{C}(\mathbb{P}^2; {\rm Sing}(\mathcal{C})) = \frac{4k - \sum_{r \geq 2}rt_{r} }{\sum_{r \geq 2}t_{r}} \geq \frac{2k - \sum_{r \geq 2, r \neq k} rt_{r}}{\sum_{r \geq 2}t_{r}-2}.$$
 (The last inequality holds if $4k - \sum_{r \geq 2}rt_{r} <0$, but if $4k - \sum_{r \geq 2}rt_{r} \geq 0$ we are done anyway).
  Let us denote these two $k$-fold points by $P,Q$. Consider the blow-up along $P,Q$ and denote by $l_{P,Q}$ the line passing through these two points:  $l_{P,Q}$ is a $(-1)$-curve.
  Now we blow down $l_{P,Q}$, which leads us to $\mathbb{P}^{1} \times \mathbb{P}^{1}$. Recall that there is a one-to-one correspondence between linear series of conics
  on $\mathbb{P}^2$ and linear series of $(1,1)$-curves on $\mathbb{P}^{1} \times \mathbb{P}^{1}$ and moreover since all $k$ conics are irreducible thus after
  passing to $\mathbb{P}^{1} \times \mathbb{P}^{1}$ the corresponding $(1,1)$-curves are also irreducible.
  The above mentioned correspondence tells us that it is enough to find a Hirzebruch type inequality for configurations of $(1,1)$-curves without $k$-fold points.
  Using Theorem \ref{Hirz11} we obtain

  $$\frac{2k-f_1}{f_{0}-2}\geq \frac{-34+2t_{2}+f_{1}}{f_{0}-2} - 8.$$
\end{itemize}
The last remaining case is $t_{k}=1$ and at this moment this is an open problem.
There are some unexpected obstacles which did not allow us to deal with this case. However, we hope to come back to this problem soon.

Now we show a Hirzebruch-type inequality for irreducible $(1,1)$-curves on $\mathbb{P}^{1} \times \mathbb{P}^{1}$.
\begin{theorem}
\label{Hirz11}
Let $\mathcal{C}_{k}$ be a configurations of $k \geq 4$ irreducible $(1,1)$-curves on $\mathbb{P}^{1} \times \mathbb{P}^{1}$ such that $t_{k} = 0$. Then one has
$$9 + k + t_{2} + t_{3} \geq \sum_{r \geq 5} (k-4)t_{r}.$$
\end{theorem}
\begin{proof}
 We will mimic the argumentation due to Hirzebruch -- let us sketch the main idea of his construction. For a configuration of irreducible $(1,1)$-curves on $\mathbb{P}^{1} \times \mathbb{P}^{1}$
 we construct an abelian covering $X$ of order $n^{k-1}$ branched along this configuration, which is singular at certain points - these singular points exactly come from
 singular points of $(1,1)$-curves configurations with multiplicities $m(P_{i}) \geq 3$. After the minimal desingularization we obtain a minimal smooth projective surface $Y$
 (minimal in the sense that there is no $(-1)$-curve) with non-negative Kodaira dimension. Over a singular point $p$ of $X$ we have in $Y$ a curve $C_{p}$ with
$$e(C_{p})=2-2g(C_{p})=n^{r-1}(2-r)+rn^{r-2},$$
by Hurwitz formula, cf (2.1) in \cite{LT}.
 Then we will use the well-known Miyaoka-Yau inequality (see for instance \cite{Miyaoka}) which tells us that $c_{1}(Y)^2 \leq 3c_{2}(Y)$ -- this leads us to the desired result.
 Since the construction of this cover is well-studied (see for instance \cite{Xavier}), thus we focus only on computing $c_{1}$ and $c_{2}$ on the minimal desingularization $Y$.
  Firstly, we compute $c_{2}$. Observe  that
  $$c_{2}(Y) = e(Y) = e(X \setminus {\rm Sing}(X)) + \sum_{r\geq 3}  n^{k-1-r}t_{r}e(C_{p}).$$
  Then
    $$e(X \setminus {\rm Sing(X)}) = n^{k-1}e(\mathbb{P}^{1} \times \mathbb{P}^{1} \setminus \mathcal{C}'_{k}) + n^{k-2}e(\mathcal{C}'_{k} \setminus {\rm Sing}(\mathcal{C}'_{k})) + n^{k-3}t_{2}.$$
It is elementary to see that
    $$e(\mathbb{P}^{1} \times \mathbb{P}^{1} \setminus \mathcal{C}'_{k}) = 4 - e(\mathcal{C}'_{k}),$$
    $$e(\mathcal{C}'_{k}) = k e(\mathbb{P}^1) - \sum (r-1)t_{r} = 2k - \sum (r-1)t_{r},$$
and
$$\sum_{r\geq 3}  n^{k-1-r}t_{r}e(C_{p}) = \sum_{r \geq 3} (n^{r-1}(2-r) + rn^{r-2})t_{r}n^{k-1-r}.$$
Finally
    $$e(Y) = n^{k-1}(4 - 2k + \sum (r-1)t_{r}) + n^{k-2}(2k-\sum rt_{r} + 2\sum t_{r} - \sum rt_{r})+n^{k-3}(t_{2} + \sum_{r \geq 3} rt_{r}).$$
Using the notation of $f_{i}$ we have
    $$\frac{1}{n^{k-3}}e(Y) = n^2(4-2k+ f_{1} - f_{0}) + 2n(k + f_{0} - f_{1}) + f_{1} -t_{2}.$$
Now we compute $c_{1}(Y)^2 = K_{Y}^2$. Here, as it is done in
\cite{LT}, we consider  the blow-up $\tau:
\widetilde{\mathbb{P}^{1} \times \mathbb{P}^{1} }\to \mathbb{P}^{1} \times \mathbb{P}^{1}$ in points $p$ of the
configurations with $m(p) \geq 3$ (in points such that at least three $(1,1)$-curves meet). Using the result \cite[Lemma~2.3]{LT} we have
$$K' = \tau^{*}(K_{\mathbb{P}^{1} \times \mathbb{P}^1}) + \sum_{m(p) \geq 3} E_{p} + \frac{n-1}{n} \bigg(\sum_{m(p)\geq3} E_{p} + \tau^{*}(\mathcal{C}'_{k}) - \sum_{m(p) \geq 3} r E_{p}\bigg)$$ and
$$K_{Y}^2 = n^{k-1}(K')^2,$$
where $E_p$ are the exceptional divisors of the blow-up $\tau$.
 Hence
$$K' = \tau^{*}(K_{\mathbb{P}^{1} \times \mathbb{P}^1}) + \frac{n-1}{n} \tau^{*}(\mathcal{C}'_{k}) + \sum_{r \geq 3} E_{r} \bigg(1 + \frac{n-1}{n}(1-r)\bigg).$$
Another elementary computations lead us to the following equality
$$ n^{2}K'^{2} = 2k^2 - \sum_{r \geq 3}t_{r}(1-r)^{2} + n\bigg(8k - 4k^{2} - \sum_{r \geq 3}t_{r}(-2r^{2}+6r-4)\bigg) + n^{2}\bigg(8 - 8k + 2k^{2} - \sum_{r\geq 3}t_{r}(r-2)^{2}\bigg).$$
Assume for now that $n \geq 2$ and $k \geq 4$. Then we need to
show that $Y$ has non-negative Kodaira dimension. Indeed, using the
same argumentation as in \cite[Section 2.3]{LT} one can see that
$K'$ is nef for $n=2$ and ample for $n \geq 3$, thus we can
apply the Miyaoka-Yau inequality for $n = 3$ and obtain (after some tedious and long computations)
$$9 + k - t_{2} \geq \sum_{r \geq 2}(r-4)t_{r}$$

and finally
$$9 + k + t_{2} + t_{3} \geq \sum_{r \geq 5} (r-4)t_{r}.$$

\end{proof}
\section{Configurations of curves and the standard Cremona transform}
In this section we exhibit two interesting line configurations which deliver very negative linear Harbourne constants, namely Klein and Wiman configurations of lines and we will see that performing the standard Cremona transform for those line configurations allows to construct two interesting conic configurations. By \emph{the linear Harbourne constant} we mean the local Harbourne constant at $\mathcal{P}$ such that the infimum is taken over line configurations on the complex projective plane.
\subsection{Klein configuration}
In \cite{BdRHHLPSz} the authors consider the linear case of the Harbourne constants. It turns out that Klein configuration delivers the second most negative value of a linear Harborune constant, which is equal to $-3$. Klein configuration consists of $21$ lines and delivers $t_{3} = 28, t_{4}=21$. Now we are going to apply the standard Cremona transform to obtain a configuration of $21$ conics. Take three non-collinear points away from the singular locus of Klein configuration. After Cremona transform based on these three point we obtain Cremona-Klein configuration $\mathcal{K}$ of $21$ conics and the following collection of singular points
$$t_{3} = 28, \ t_{4} = 21, \ t_{21} = 3.$$
Then the conical local Harbourne constant at ${\rm Sing}(\mathcal{K})$ is equal to
$$H_{C}(\mathbb{P}^{2}; {\rm Sing}(\mathcal{K})) = \frac{ 4\cdot21^2 - 3\cdot21^2 - 28\cdot9 - 21\cdot16}{49+3} = -2.827.$$

\subsection{Wiman configuration}
Wiman configuration of lines delivers, according to the best of
our knowledge, the most negative linear Harbourne constant, which
is equal to $-3.36$. This configuration consists of $45$ lines and
delivers the following collection of singular points
$$t_{3}=120,\ t_{4}=45,\ t_{5}=36.$$
We mimic our argumentation from the previous subsection. Let us choose three non-collinear points away from singular locus. After application of the standard Cremona transform we obtain Cremona-Wiman configuration $\mathcal{W}$ of $45$ conics and the following singular points
$$t_{3} = 120,\ t_{4}=45,\ t_{5}=36,\ t_{45} = 3.$$
Then the conical local Harbourne constant at ${\rm Sing}(\mathcal{W})$ is equal to
$$H_{C}(\mathbb{P}^{2}; {\rm Sing}(\mathcal{W})) = \frac{ 4 \cdot 45^2 - 3\cdot45^2 - 120\cdot9 - 45\cdot16 - 36\cdot25}{201+3} = -3.31.$$
\subsection{General remark}
At the end of this note let us point out a general remark due to X. Roulleau.
\begin{remark}
Let $\mathcal{C}$ be a configuration of curves on $\mathbb{P}^{2}_{\mathbb{C}}$ with $s$-singular points and denote by $H_{1}$ the local $H$-constant. Let us perform the standard Cremona transform on three general points. Denote by $\mathcal{C}'$ the curve configuration after the standard Cremona transform and denote by $H_{2}$ the local $H$-constant. Then we have
$$H_{2} = \frac{s}{s+3} H_{1}.$$
In particular, applying this construction $n\geq 1$ times to the elliptic curve configuration from \cite[Proposition~5]{Xavier} we obtain a configuration of degree $3\cdot2^{n}$ curves which asymptotically has the same local $H$-constant equal to $-4$.
\end{remark}

\section*{Acknowledgement}
The authors would like to express their gratitude to Xavier Roulleau for a proofreading of this manuscript and useful remarks and to Tomasz Szemberg for
stimulating conversations. The first author is partially supported by National Science Centre Poland Grant 2014/15/N/ST1/02102 and the second author is
supported by National Science Centre Poland Grant 2014/15/B/ST1/02197.



\bigskip
   Piotr Pokora,
   Instytut Matematyki,
   Pedagogical University of Cracow,
   Podchor\c a\.zych 2,
   PL-30-084 Krak\'ow, Poland.

\nopagebreak
   \textit{E-mail address:} \texttt{piotrpkr@gmail.com}

   \bigskip
   Halszka Tutaj-Gasi\'nska,
   Institute of Mathematics,
   Jagiellonian University,
   \L ojasiewicza 6,
   PL-30-348 Krak\'ow, Poland.

\nopagebreak
   \textit{E-mail address:} \texttt{Halszka.Tutaj@im.uj.edu.pl}


\end{document}